\documentclass{amsart}

 \newtheorem{thm}{Theorem}[section]
 \newtheorem{cor}[thm]{Corollary}

 \newtheorem{Opp}[thm]{Open Problem}
 \theoremstyle{definition}
 \newtheorem{defn}[thm]{Definition}
 \theoremstyle{remark}

 \numberwithin{equation}{section}

\begin{document}
%
%
%
%
%
%
%
%
%

\title[On the stability of the modified entropy equation]
 {On the stability of the modified entropy\\ equation}
\author[Eszter Gselmann]{Eszter Gselmann}

\address{
Institute of Mathematics\\
University of Debrecen\\
P.~O.~Box: 12.\\
Debrecen\\
Hungary\\
H--4010}

\email{gselmann@math.klte.hu}

\thanks{This research has been supported by the Hungarian Scientific Research Fund
(OTKA) Grant NK 68040 and also by the Universitas Scholarship founded by the K\&H Bank
(Kereskedelmi \'{e}s Hitelbank Rt.).}



\date{\today}

\begin{abstract}
The aim of this paper is to prove that under some conditions
the modified entropy equation is
stable on its one--dimensional domain.
\end{abstract}

\maketitle

\section{Introduction}

The basic problem in the stability theory of functional equations is
whether an approximate solution of a functional equation
can be approximated by a
solution of the equation in question.
In case of an affirmative answer to the previous problem we say that
the investigated functional equation is \emph{stable}.
About the stability theory of functional equations one can read e.g.,
in Czerwik \cite{Cze02}, Forti \cite{For95}, Hyers--Isac--Rassias \cite{HIR98},
Jung \cite{Jun01}, Moszner \cite{Mos04} Rassias \cite{Ras03}.

In this paper we investigate the above problem concerning
the functional equation
\begin{equation}\label{m_eq}
f(x, y, z)=f(x, y+z, 0)+(y+z)^{\alpha}f\left(0, \frac{y}{y+z}, \frac{z}{y+z}\right),
\end{equation}
where $x, y, z$ are positive real numbers and $\alpha$ is a given real number.
Equation \eqref{m_eq} is a special case of the
so--called \emph{modified entropy equation},
\begin{equation}\label{Eq_me_eq}
f(x, y, z)=f(x, y+z, 0)+\mu(y+z)f\left(0, \frac{y}{y+z}, \frac{z}{y+z}\right),
\end{equation}
where $\mu$ is a given multiplicative function defined on the positive cone of $\mathbb{R}^{k}$ and
\eqref{Eq_me_eq} is supposed to hold for all elements $x, y, z$ of the above mentioned cone
and all operations on vectors are to be understood componentwise.
The symmetric solutions of equation \eqref{Eq_me_eq} were determined in \cite{Gse08}
(see also \cite{AGMS08}).

Throughout this paper we will use the following notations
\[
\mathbb{R}_{+}=\left\{x\in\mathbb{R} \mid x\geq 0\right\} \quad
\text{and} \quad
\mathbb{R}_{++}=\left\{x\in\mathbb{R} \mid x>0\right\},
\]
where $\mathbb{R}$ denotes the set of the real numbers and $\mathbb{N}$ stands for the set
of the positive integers.

By a real interval we always mean a subinterval of positive length of $\mathbb{R}$.
Furthermore, in case $U$ and $V$ are real intervals, then their sum
\[
U+V=\left\{u+v \mid u\in U, v\in V\right\}
\]
is obviously a real interval, as well.

In the subsequent sections the definition of logarithmic and multiplicative functions
will be needed. (See also Kuczma \cite{Kuc85} and Rad\'{o}--Baker \cite{RB87}).

\begin{defn}
Let $I\subset \mathbb{R}_{++}$ be a set and
\[
\mathcal{I}=\left\{(x, y)\in\mathbb{R}^{2}_{++}\mid x, y, xy\in I\right\}.
\]
A function $\mu:I\rightarrow\mathbb{R}$ is called multiplicative on
$\mathcal{I}$ if
\[
\mu(xy)=\mu(x)\mu(y)
\]
holds for all $(x, y)\in \mathcal{I}$.

Furthermore, we say that a function $l:I\rightarrow\mathbb{R}$
is logarithmic on $\mathcal{I}$ if
\[
l(xy)=l(x)+l(y)
\]
holds for all $(x, y)\in\mathcal{I}$.
\end{defn}

\section{Preliminaries}
During the proof of our main result the stability of a simple associativity
equation will be used which is contained in the following.
\begin{thm}\label{Thm2.1}
Let $U, V, W$ be real intervals,
$A:(U+V)\times W\rightarrow\mathbb{R}$,
$B:U\times (V+W)\rightarrow\mathbb{R}$
and suppose that
\begin{equation}\label{Eq1.1}
\left|A(u+v, w)-B(u, v+w)\right|\leq \varepsilon
\end{equation}
holds for all $u\in U$, $v\in V$ and $w\in W$.
Then there exists a function $\varphi:U+V+W\rightarrow\mathbb{R}$ such that
\begin{equation}\label{Eq1.2}
\left|A(p, q)-\varphi(p+q)\right|\leq 2\varepsilon \quad \left(p\in (U+V), q\in W\right)
\end{equation}
and
\begin{equation}\label{Eq1.3}
\left|B(t, s)-\varphi(t+s)\right|\leq \varepsilon \quad \left(t\in U, s\in (V+W)\right)
\end{equation}
hold.
\end{thm}

\begin{proof}
First we prove that inequalities \eqref{Eq1.2} and \eqref{Eq1.3} hold for
 compact intervals. Therefore suppose that $U, V, W$ are compact real intervals, that is,
\[
U=[u_{1}, u_{2}], \quad V=[v_{1}, v_{2}] \quad \text{and} \quad W=[w_{1}, w_{2}].
\]
Choose the intervals
\[
W^{(k)}=[w_{1, k}, w_{2, k}] \quad \left(k=1, \ldots, N\right)
\]
such that
\begin{equation}\label{Eq1.4}
\max_{1\leq k\leq N}\left(w_{2, k}-w_{1, k}\right) < v_{2}-v_{1},
\end{equation}
\begin{equation}\label{Eq1.5}
\bigcup^{N}_{k=1}W^{(k)}=W
\end{equation}
and
\begin{equation}\label{Eq1.6}
w_{1, k}<w_{1, k+1}< w_{2, k}<w_{2, k+1}
\end{equation}
should hold.

For $i\in\left\{1, 2\right\}$ let
\begin{equation}\label{Eq1.7}
\varphi_{i, k}\left(\xi\right)=A\left(\xi-w_{i, k}, w_{i, k}\right).  \quad \left(\xi \in U+V+w_{i, k}\right)
\end{equation}
Then
\[
\left|\varphi_{i, k}(t+s)-B(t, s)\right|
=\left|A\left(t+s-w_{i, k}, w_{i, k}\right)-B(t, s)\right| \leq \varepsilon
\]
holds for all $t\in U$, $s\in V+w_{i, k}$ and $i\in\left\{1, 2\right\}$.

Therefore
\begin{multline}\label{Eq1.8}
\left|\varphi_{1, k}(t+s)-\varphi_{2, k}(t+s)\right| \\
\leq \left|\varphi_{1, k}(t+s)-B(t, s)\right|+
\left|\varphi_{2, k}(t+s)-B(t, s)\right| \leq 2\varepsilon
\end{multline}
if $t\in U$ and $s\in (V+w_{1, k})\cap (V+w_{2, k})$.

Due to \eqref{Eq1.6}
\[
(V+w_{1, k})\cap (V+w_{2, k})=
[v_{1}+w_{2, k}, v_{2}+w_{1, k}]
\]
and so
\[
U+\left((V+w_{1, k})\cap (V+w_{2, k})\right)=
\left(U+V+w_{1, k}\right)\cap \left(U+V+w_{2, k}\right).
\]
Thus by \eqref{Eq1.8}
\[
\left|\varphi_{1, k}(\xi)-\varphi_{2, k}(\xi)\right| \leq 2\varepsilon
\]
if $\xi\in \left(U+V+w_{1, k}\right)\cap \left(U+V+w_{2, k}\right)$
and $k\in \left\{1, \ldots, N\right\}$.
Let
\[
\varphi_{k}(\xi)=
\left\{
\begin{array}{lcl}
\varphi_{1, k}(\xi), &\text{if}& \xi\in \left(U+V+w_{1, k}\right)\setminus \left(U+V+w_{2, k}\right) \\
\varphi_{2, k}(\xi), &\text{if}& \xi\notin \left(U+V+w_{1, k}\right)\setminus \left(U+V+w_{2, k}\right).
\end{array}
\right.
\]
The function $\varphi_{k}$ is defined on $U+V+W^{(k)}$, since by \eqref{Eq1.4}
\[
\left(U+V+w_{1, k}\right)\cup \left(U+V+w_{2, k}\right)=U+V+W^{(k)}.
\]
Additionally,
\[
\left|\varphi_{k}(t+s)-B(t, s)\right|\leq \varepsilon
\]
for all $t\in U$, $s\in V+W^{(k)}$ and
$k\in \left\{1, \ldots, N\right\}$.
This shows that
\[
\left|\varphi_{k}(\xi)-\varphi_{k+1}(\xi)\right|\leq 2\varepsilon,
\]
if $\xi\in \left(U+V+W^{(k)}\right)\cap \left(U+V+W^{(k+1)}\right)$, and
$k\in \left\{1, \ldots, N-1\right\}$.
Define the function $\varphi$ on $U+V+W$ by
\[
\varphi(\xi)=\left\{
\begin{array}{lcl}
\varphi_{k}(\xi), &\text{if} &
\xi\in \left(U+V+W^{(k})\right)\setminus \bigcup^{N}_{k\neq i=1} \left(U+V+W^{(i)}\right) \\
\varphi_{\max\left\{i, j\right\}}(\xi), &\text{if} &
\xi\in \left(U+V+W^{(i)}\right)\cap \left(U+V+W^{(j)}\right), i, j\in \left\{1, \ldots, N\right\}.
\end{array}
\right.
\]
Then
\begin{equation}\label{Eq1.9}
\left|\varphi(t+s)-B(t, s)\right|\leq \varepsilon
\end{equation}
holds for all $t\in U$ and $s\in V+W$.

Let now $p\in U+V$ and $q\in W$. Then
$p=u+v$ with certain $u\in U$ and $v\in V$, therefore
$A(p, q)=A(u+v, q)$ and inequality \eqref{Eq1.1} yields that
\[
\left|A(u+v, q)-B(u, v+q)\right|\leq \varepsilon,
\]
furthermore, due to \eqref{Eq1.9}
\[
\left|B(u, v+q)-\varphi(u+v+q)\right|\leq \varepsilon,
\]
therefore
\begin{multline}\label{Eq1.10}
\left|A(p, q)-\varphi(p+q)\right|\\
\leq \left|A(u+v, q)-B(u, v+q)\right|+\left|B(u, v+q)-\varphi(u+v+q)\right|\\
\leq 2\varepsilon
\end{multline}
holds for all $p\in U+V$ and $q\in W$.

Finally, we will prove that \eqref{Eq1.2} and \eqref{Eq1.3} hold also
in case the intervals $U, V$ and $W$ are not necessarily compact.

Indeed, there exist sequences $(U_{n}), (V_{n})$ and $(W_{n})$ of compact
real intervals such that
\[
U_{n}\subset U_{n+1}, \quad V_{n}\subset V_{n+1} \quad \text{and}
\quad W_{n}\subset W_{n+1}
\]
for all $n\in \mathbb{N}$ and
\[
U=\bigcup^{\infty}_{n=1}U_{n}, \quad V=\bigcup^{\infty}_{n=1}V_{n} \quad
\text{and} \quad W=\bigcup^{\infty}_{n=1}W_{n}.
\]
From the previous part of the proof we get that, for all
$n\in\mathbb{N}$, there exists a function
$\varphi_{n}:U_{n}+V_{n}+W_{n}\rightarrow\mathbb{R}$ such that
\[
\left|A(p, q)-\varphi_{n}(p+q)\right|\leq 2\varepsilon \quad
\left(p\in U_{n}+V_{n}, q\in W_{n}\right)
\]
and
\[
\left|B(t, s)-\varphi_{n}(t+s)\right|\leq \varepsilon .
\quad
\left(t\in U_{n}, s\in V_{n}+W_{n}\right)
\]
Let
\[
\varphi(x)=\varphi_{1}(x),
\]
if $x\in U_{1}+V_{1}+W_{1}$ and
\[
\varphi(x)=\varphi_{n}(x)
\]
if $x\in \left(U_{n}+V_{n}+W_{n}\right)\setminus \left(U_{n-1}+V_{n-1}+W_{n-1}\right)$ for all
$n\in\mathbb{N}$.

Then the function $\varphi$ is well--defined on the set $U+V+W$ and
inequalities \eqref{Eq1.2} and \eqref{Eq1.3} hold, indeed.
\end{proof}
With the choice $\varepsilon_{1}=\varepsilon_{2}=0$, we get the following theorem.
Nevertheless, it was proved in Maksa \cite{Mak00}.
\begin{cor}
Let $U, V$ and $W$ be real intervals,
$A:(U+V)\times W\rightarrow\mathbb{R}$,
$B:U\times (V+W)\rightarrow\mathbb{R}$ and suppose that
\[
A(u+v, w)=B(u, v+w)
\]
holds for all $u\in U, v\in V$ and $w\in W$.
Then there exists a function $\varphi:U+V+W\rightarrow\mathbb{R}$
such that
\begin{equation}
A(p, q)=\varphi(p+q)
\end{equation}
for all $p\in U+V$ and $q\in W$ and
\begin{equation}
B(t, s)=\varphi(t+s)
\end{equation}
for all $t\in U$ and $s\in V+W$.
\end{cor}

The following three theorems say that the parametric fundamental equation of information is
hyperstable in case the parameter is negative, stable in case the parameter is zero and
it is superstable in case the parameter is positive but does not equal to one,
respectively.

\begin{thm}\label{Thm2.2} \cite{Gse09}
Let $\alpha, \varepsilon\in\mathbb{R}$, $\alpha<0$, $\varepsilon\geq 0$
and $f:]0, 1[\rightarrow\mathbb{R}$ be a function.
Assume that
\begin{equation}
\left|f(x)+(1-x)^{\alpha}f\left(\frac{y}{1-x}\right)-
f(y)-(1-y)^{\alpha}f\left(\frac{x}{1-y}\right)\right|\leq \varepsilon
\end{equation}
holds for all
$(x, y)\in D^{\circ}=\left\{(x, y)\in\mathbb{R}^{2}\mid x, y, x+y\in ]0, 1[\right\}$.
Then, and only then, there exist $a, b\in\mathbb{R}$ such that
\begin{equation}
f(x)=ax^{\alpha}+b(1-x)^{\alpha}-b
\end{equation}
for all $x\in ]0, 1[$.
\end{thm}

\begin{thm}\label{Thm2.3} \cite{GM08}
Let $\varepsilon\geq 0$
and $f:]0, 1[\rightarrow\mathbb{R}$ be a function.
Assume that
\begin{equation}
\left|f(x)+f\left(\frac{y}{1-x}\right)-
f(y)-f\left(\frac{x}{1-y}\right)\right|\leq \varepsilon
\end{equation}
holds for all $(x, y)\in D^{\circ}$.
Then there exists a logarithmic function
$l: ]0,1[\rightarrow\mathbb{R}$ and $c\in \mathbb{R}$ such that
\begin{equation}\label{Eq2.3}
\left|f(x)-\left[l(1-x)+c\right]\right|\leq 63 \varepsilon
\end{equation}
holds for all $x\in ]0,1[$.
\end{thm}

\begin{thm}\label{Thm2.4} \cite{Gse09a}
Let $\alpha, \varepsilon\in\mathbb{R}$, $1\neq\alpha>0$, $\varepsilon\geq 0$
and $f:]0, 1[\rightarrow\mathbb{R}$ be a function.
Assume that
\begin{equation}
\left|f(x)+(1-x)^{\alpha}f\left(\frac{y}{1-x}\right)-
f(y)-(1-y)^{\alpha}f\left(\frac{x}{1-y}\right)\right|\leq \varepsilon
\end{equation}
holds for all
$(x, y)\in D^{\circ}$.
Then there exist $a, b\in\mathbb{R}$ such that
\begin{equation}
\left|f(x)-\left[ax^{\alpha}+b(1-x)^{\alpha}-b\right]\right|\leq K(\alpha)\varepsilon
\end{equation}
holds for all $x\in ]0, 1[$, where
\[
K(\alpha)=
\left|2^{1-\alpha}-1\right|^{-1}
\left(3+12\cdot2^{\alpha}+\frac{32\cdot3^{\alpha+1}}{\left|2^{-\alpha}-1\right|}\right).
\]
\end{thm}

\section{The main result}
In view of the results of the last section we are able to prove our main theorem,
which is contained in the following.

\begin{thm}\label{Thm3.1}
Let $\alpha, \varepsilon\in\mathbb{R}$, $\alpha\neq 1, \varepsilon\geq 0$
and $f:\mathbb{R}^{3}_{+}\rightarrow\mathbb{R}$ be a function.
Assume that
\begin{equation}\label{Eq1.13}
\left|f(x, y, z)-f(x, y+z, 0)-(y+z)^{\alpha}f\left(0, \frac{y}{y+z}, \frac{z}{y+z}\right)\right|\leq \varepsilon_{1}
\end{equation}
and
\begin{equation}\label{Eq1.14}
\left|f(x, y, z)-f\left(\sigma(x), \sigma(y), \sigma(z)\right)\right|\leq \varepsilon_{2}
\end{equation}
hold for all $x, y, z\in\mathbb{R}_{++}$ and for all permutations
$\sigma:\left\{x, y, z\right\}\rightarrow\left\{x, y, z\right\}$.

Then, in case $\alpha<0$, there exist $a\in\mathbb{R}$ and a function
$\varphi_{1}:\mathbb{R}_{++}\rightarrow\mathbb{R}$ such that
\begin{equation}
\left|f(x, y, z)-\left[ax^{\alpha}+ay^{\alpha}+az^{\alpha}+\varphi_{1}(x+y+z)\right]\right|\leq
2\varepsilon_{1}+3\varepsilon_{2}
\end{equation}
holds for all $x, y, z\in\mathbb{R}_{++}$.

Furthermore, if $\alpha=0$, then there exists a function $\varphi_{2}:\mathbb{R}_{++}\rightarrow\mathbb{R}$
such that
\begin{equation}
\left|f(x, y, z)-\varphi_{2}(x+y+z)\right|\leq 191\varepsilon_{1}+1263\varepsilon_{2}
\end{equation}
holds for all $x, y, z\in\mathbb{R}_{++}$.

Finally, if $1\neq \alpha>0$, then for all $n\in\mathbb{N}$, there exists a function
$\psi_{n}:]0, 3n]\rightarrow\mathbb{R}$ such that
\[
\left|f(x, y, z)-\left[ax^{\alpha}+ay^{\alpha}+az^{\alpha}+\psi_{n}(x+y+z)\right]\right|
\leq c_{n}(\alpha)\varepsilon_{n}+d_{n}(\alpha)\varepsilon_{2}
\]
holds for all $x, y, z\in ]0, n]$, where
\[
c_{n}(\alpha)=2+7\cdot 2^{\alpha}n^{\alpha}K(\alpha)
\quad
\text{and}
\quad
d_{n}(\alpha)=4+7\cdot 2^{\alpha+2}n^{\alpha}K(\alpha).
\]
\end{thm}

\begin{proof}
Assume that inequalities \eqref{Eq1.13} and \eqref{Eq1.14} hold.
Define the functions $F$ and $h$ on $\mathbb{R}^{2}_{++}$ and on $]0, 1[$
respectively, by
\[
F(u, v)=f(0, u, v) \quad \left(u, v \in\mathbb{R}_{++}\right)
\]
and by
\[
h(t)=F(1-t, t). \quad \left(t\in ]0, 1[\right)
\]

Then inequality \eqref{Eq1.14} immediately implies that
\[
\left|F(u, v)-F(v, u)\right|\leq \varepsilon_{2}
\quad \left(u, v \in\mathbb{R}_{++}\right)
\]
and
\[
\left|h(t)-h(1-t)\right|\leq \varepsilon_{2}
\quad \left(t\in ]0, 1[\right)
\]
holds.
Furthermore, in view of the definitions of the functions $F$ and $h$
\eqref{Eq1.13} yields that
\begin{equation}\label{Eq1.15}
\left|f(x, y, z)-F(x, y+z)-(y+z)^{\alpha}h\left(\frac{y}{y+z}\right)\right|\leq
\varepsilon_{1}+\varepsilon_{2}
\end{equation}
holds for all $x, y, z\in \mathbb{R}_{++}$.
Interchanging $x$ and $y$ in \eqref{Eq1.15} and using inequality
\eqref{Eq1.14} we obtain that
\begin{equation}\label{Eq1.16}
\left|F(x, y+z)+(y+z)^{\alpha}h\left(\frac{y}{y+z}\right)
-F(y, x+z)-(x+z)^{\alpha}h\left(\frac{x}{x+z}\right)\right|\leq \varepsilon_{1}+2\varepsilon_{2}
\end{equation}
for all $x, y, z\in \mathbb{R}_{++}$. If we substitute into this
$z=1-x-y$, where $(x, y)\in D^{\circ}$, then we get that
\begin{equation}\label{Eq1.17}
\left|F(x, 1-x)+(1-x)^{\alpha}h\left(\frac{y}{1-x}\right)-
F(y, 1-y)-(1-y)^{\alpha}h\left(\frac{x}{1-y}\right)\right|
\leq \varepsilon_{1}+2\varepsilon_{2}.
\end{equation}
Using two times the approximate symmetry, we get that
\begin{equation}\label{Eq1.18}
\left|h(x)+(1-x)^{\alpha}h\left(\frac{y}{1-x}\right)
-h(y)-(1-y)^{\alpha}h\left(\frac{x}{1-y}\right)\right|\leq \varepsilon_{1}+4\varepsilon_{2}
\end{equation}
holds for all $(x, y)\in D^{\circ}$.

Concerning the parameter $\alpha$, we have to distinguish three cases.
In the first case assume that $\alpha<0$.
Then due to Theorem \ref{Thm2.2}., there exist $a, b\in\mathbb{R}$ such that
\[
h(x)=ax^{\alpha}+b(1-x)^{\alpha}-b.
\quad
\left(x\in ]0, 1[\right)
\]
On the other hand inequality
\[
\left|h(t)-h(1-t)\right|\leq \varepsilon_{2}
\quad
\left(t\in ]0, 1[\right)
\]
implies that
\[
\left|(a-b)(x^{\alpha}-(1-x)^{\alpha})\right|\leq \varepsilon_{2},
\quad
\left(x\in ]0, 1[\right)
\]
that is,
\[
\left|a-b\right|\leq \frac{\varepsilon_{2}}{\left|x^{\alpha}-(1-x)^{\alpha}\right|}
\]
holds for all $x\in ]0, 1[$.
Taking the limit $x\rightarrow 0+$, we obtain that $a=b$, since $\alpha<0$.
Therefore
\[
h(x)=ax^{\alpha}+a(1-x)^{\alpha}-a,
\quad
\left(x\in ]0, 1[\right)
\]
in case $\alpha<0$.

Define the function $H$ on $\mathbb{R}^{2}_{++}$ by
\[
H(x, y)=F(x, y)-ax^{\alpha}-ay^{\alpha},
\quad \left(x, y\in\mathbb{R}_{++}\right)
\]
then
\begin{multline*}
\left|H(x, y+z)-H(y, x+z)\right| \\
=\left|F(x, y+z)-ax^{\alpha}-a(y+z)^{\alpha}
-F(y, x+z)+ay^{\alpha}+a(x+z)^{\alpha}\right| \\
=\left|F(x, y+z)+ay^{\alpha}+az^{\alpha}-a(y+z)^{\alpha}
-F(y, x+z)-ax^{\alpha}-az^{\alpha}+a(x+z)^{\alpha}\right| \\
=\left|F(x, y+z)+(y+z)^{\alpha}h\left(\frac{y}{y+z}\right)
-F(y, x+z)-(x+z)^{\alpha}h\left(\frac{x}{x+z}\right)\right| \\
\leq\varepsilon_{1}+2\varepsilon_{2}
\end{multline*}
holds for all $x, y, z\in\mathbb{R}_{++}$, where we used inequality (\ref{Eq1.16}) and the
fact that the function $h$ is of the form
\[
h(x)=ax^{\alpha}+a(1-x)^{\alpha}-a.
\quad \left(x\in ]0, 1[\right)
\]
Now Theorem \ref{Thm2.1}. yields, with the choice
\[
A(x, y)=H(y, x) \quad \text{and} \quad B(x, y)=H(x, y), \quad
\left(x, y\in\mathbb{R}_{++}\right)
\]
that there exists a function
$\varphi_{1}:\mathbb{R}_{++}\rightarrow\mathbb{R}$ such that
\[
\left|H(x, y)-\varphi_{1}(x+y)\right|\leq \varepsilon_{1}+2\varepsilon_{2},
\quad \left(x, y\in\mathbb{R}_{++}\right)
\]
that is,
\[
\left|F(x, y)-ax^{\alpha}-ay^{\alpha}-\varphi_{1}(x+y)\right| \leq \varepsilon_{1}+2\varepsilon_{2}.
\quad \left(x, y\in\mathbb{R}_{++}\right)
\]
Finally, \eqref{Eq1.15}, the last inequality and the triangle
inequality imply that
\begin{multline*}
\left|f(x, y, z)-\left[ax^{\alpha}+ay^{\alpha}+az^{\alpha}+\varphi_{1}(x+y+z)\right]\right| \\
\leq
\left|f(x, y, z)-F(x, y+z)-(y+z)^{\alpha}h\left(\frac{y}{y+z}\right)\right| \\
+\left|F(x, y+z)-ax^{\alpha}-a(y+z)^{\alpha}-\varphi_{1}(x+y+z)\right| \\
+\left|(y+z)^{\alpha}h\left(\frac{y}{y+z}\right)-ay^{\alpha}-az^{\alpha}+a(y+z)^{\alpha}\right| \\
\leq (\varepsilon_{1}+\varepsilon_{2})+(\varepsilon_{1}+2\varepsilon_{2})=
2\varepsilon_{1}+3\varepsilon_{2}
\end{multline*}
for all $x, y, z\in\mathbb{R}_{++}$.

Now, let us consider the case $\alpha=0$.
Applying Theorem \ref{Thm2.3}.
we obtain that there exist a logarithmic function
$l:\mathbb{R}_{++}\rightarrow\mathbb{R}$ and $c\in\mathbb{R}$
such that
\begin{equation}\label{Eq3.9}
\left|h(x)-\left[l(1-x)-c\right]\right|\leq 63\left(\varepsilon_{1}+4\varepsilon_{2}\right)
\end{equation}
holds for all $x\in ]0, 1[$.
The approximate symmetry to the point $\frac{1}{2}$ implies however that
\[
\left|l(x)-l(1-x)\right|\leq 126 (\varepsilon_{1}+4\varepsilon_{2})+\varepsilon_{2}.
\quad
\left(x\in ]0, 1[\right)
\]
Let $x=\frac{p}{p+1}$ ($p\in \mathbb{R}_{++}$) in the previous inequality.
Then this implies that
\[
\left|l(p)\right|\leq 126 (\varepsilon_{1}+4\varepsilon_{2})+\varepsilon_{2},
\quad \left(p\in \mathbb{R}_{++}\right)
\]
that is, $l:\mathbb{R}_{++}\rightarrow\mathbb{R}$ is a bounded logarithmic
function.
Although, only the identically zero function has these properties.
Therefore \eqref{Eq3.9} yields that
\[
\left|h(x)-c\right|\leq 63\left(\varepsilon_{1}+4\varepsilon_{2}\right)
\]
for all $x\in ]0, 1[$.

Define the function $H$ on $\mathbb{R}_{++}$ by
\[
H(x, y)=F(x, y)-c,
\quad \left(x, y\in \mathbb{R}_{++}\right)
\]
then
\begin{multline*}
\left|H(x, y+z)-H(y, x+z)\right|
=\left|F(x, y+z)-c-F(y, x+z)+c\right| \\
\leq
\left|F(x, y+z)+h\left(\frac{y}{y+z}\right)
-F(y, x+z)-h\left(\frac{x}{x+z}\right)\right|\\
+\left|h\left(\frac{y}{y+z}\right)-c\right|+
\left|h\left(\frac{x}{x+z}\right)-c\right| \\
\leq(\varepsilon_{1}+2\varepsilon_{2})+2\cdot 63(\varepsilon_{1}+4\varepsilon_{2})=
127\varepsilon_{1}+1010\varepsilon_{2}
\end{multline*}
for all $x, y, z\in\mathbb{R}_{++}$.

Due to Theorem \ref{Thm2.1}., there exists a function
$\varphi_{2}:\mathbb{R}_{++}\rightarrow\mathbb{R}$ such that
\[
\left|H(x, y)-\varphi_{2}(x+y)\right|\leq 127\varepsilon_{1}+1010\varepsilon_{2}.
\quad \left(x, y\in\mathbb{R}_{++}\right)
\]
Again, inequality \eqref{Eq1.15}, the definition of the function $H$ and
the triangle inequality yield that
\begin{multline*}
\left|f(x, y, z)-\varphi_{2}(x+y+z)\right| \\
\leq
\left|f(x, y, z)-F(x, y+z)-h\left(\frac{y}{y+z}\right)\right|\\
+\left|F(x, y+z)-\varphi_{2}(x+y+z)-c\right|+
\left|h\left(\frac{y}{y+z}\right)-c\right| \\
\leq (\varepsilon_{1}+\varepsilon_{2})+(127\varepsilon_{1}+1010\varepsilon_{2})+
63(\varepsilon_{1}+4\varepsilon_{2})=
191\varepsilon_{1}+1263\varepsilon_{2}
\end{multline*}
for all $x, y, z\in\mathbb{R}_{++}$.

Finally, let us consider the case $1\neq \alpha>0$.
Inequality \eqref{Eq1.18} implies that there exist
$a, b\in\mathbb{R}$ such that
\[
\left|h(x)-\left[ax^{\alpha}-b(1-x)^{\alpha}-b\right]\right|\leq
K(\alpha)(\varepsilon_{1}+4\varepsilon_{2}).
\quad \left(x\in ]0, 1[\right)
\]
Furthermore, this inequality with
\[
\left|h(t)-h(1-t)\right|\leq \varepsilon_{2}
\quad \left(t\in ]0, 1[\right)
\]
implies that
\begin{multline*}
\left|(a-b)(x^{\alpha}-(1-x)^{\alpha})\right| \\
\leq
\left|h(x)-\left[ax^{\alpha}-b(1-x)^{\alpha}-b\right]\right|\\
+\left|h(1-x)-\left[a(1-x)^{\alpha}-bx^{\alpha}-b\right]\right|+
\left|h(x)-h(1-x)\right| \\
\leq 2K(\alpha)(\varepsilon_{1}+4\varepsilon_2)+\varepsilon_{2},
\quad \left(x\in ]0, 1[\right)
\end{multline*}
Now taking the limit $x\rightarrow 1-$ , we obtain that
\[
\left|a-b\right|\leq 2K(\alpha)(\varepsilon_{1}+4\varepsilon_2)+\varepsilon_{2}.
\]
Therefore,
\begin{multline*}
\left|h(x)-\left[ax^{\alpha}+a(1-x)^{\alpha}-a\right]\right| \\
\leq
\left|h(x)-\left[ax^{\alpha}+b(1-x)^{\alpha}-b\right]\right|
+\left|a-b\right|\cdot\left|(1-x)^{\alpha}-1\right| \\
\leq K(\alpha)(\varepsilon_{1}+4\varepsilon_{2})+
2K(\alpha)(\varepsilon_{1}+4\varepsilon_{2})+\varepsilon_{2}=
3K(\alpha)(\varepsilon_{1}+4\varepsilon_{2})+\varepsilon_{2}
\end{multline*}
holds for all $x\in ]0, 1[$.

At this point of the proof we define the function $H$ on
$\mathbb{R}^{2}_{++}$ similarly as in case $\alpha<0$, that is, by
\[
H(x, y)=F(x, y)-ax^{\alpha}-ay^{\alpha},
\quad \left(x, y\in\mathbb{R}_{++}\right)
\]
to obtain
\begin{multline}
\left|H(x, y+z)-H(y, x+z)\right| \\
=\left|F(x, y+z)-ax^{\alpha}-a(y+z)^{\alpha}
-F(y, x+z)+ay^{\alpha}+a(x+z)^{\alpha}\right| \\
=\left|F(x, y+z)+ay^{\alpha}+az^{\alpha}-a(y+z)^{\alpha}
-F(y, x+z)-ax^{\alpha}-az^{\alpha}+a(x+z)^{\alpha}\right| \\
\leq\left|F(x, y+z)+(y+z)^{\alpha}h\left(\frac{y}{y+z}\right)
-F(y, x+z)-(x+z)^{\alpha}h\left(\frac{x}{x+z}\right)\right|\\
+\left|(y+z)^{\alpha}h\left(\frac{y}{y+z}\right)-\left[ay^{\alpha}+az^{\alpha}-a(y+z)^{\alpha}\right]\right| \\
+\left|(x+z)^{\alpha}h\left(\frac{x}{x+z}\right)-\left[ax^{\alpha}+az^{\alpha}-a(x+z)^{\alpha}\right]\right| \\
\leq (\varepsilon_{1}+2\varepsilon_{2})
+(y+z)^{\alpha}(3K(\alpha)(\varepsilon_{1}+4\varepsilon_{2})+\varepsilon_{2})
+(x+z)^{\alpha}(3K(\alpha)(\varepsilon_{1}+4\varepsilon_{2})+\varepsilon_{2})\\
=(\varepsilon_{1}+2\varepsilon_{2})
+\left[(x+z)^{\alpha}+(y+z)^{\alpha}\right](3K(\alpha)(\varepsilon_{1}+4\varepsilon_{2})+\varepsilon_{2}),
\quad \left(x, y, z\in\mathbb{R}_{++}\right)
\end{multline}
that is, the right hand side of the inequality depends
on the variables $x, y, z$.
Thus Theorem \ref{Thm2.1}. cannot directly be applied.
Let however $n\in\mathbb{N}$ arbitrarily fixed.
Then the previous inequality implies that
\[
\left|H(x, y+z)-H(y, x+z)\right|\leq
\left(1+3\cdot 2^{\alpha+1}n^{\alpha}K(\alpha)\right)\varepsilon_{1}+
\left(3+3\cdot 2^{\alpha+3}n^{\alpha}K(\alpha)\right)\varepsilon_{2}
\]
holds for all $x, y, z\in ]0, n]$.
For the sake of brevity, let us introduce the notations
\[
a_{n}(\alpha)=1+3\cdot 2^{\alpha+1}n^{\alpha}K(\alpha)
\quad
\text{and}
\quad
b_{n}(\alpha)=3+3\cdot 2^{\alpha+3}n^{\alpha}K(\alpha).
\]
Due to Theorem \ref{Thm2.1}. there exists a function
$\psi_{n}:]0, 3n]\rightarrow\mathbb{R}$ such that
\begin{equation}
\left|H(x, y)-\psi_{n}(x+y)\right|\leq
a_{n}(\alpha)\varepsilon_{n}+b_{n}(\alpha)\varepsilon_{2}
\end{equation}
holds for all $x, y\in ]0, n]$.
Using the definition of the function $H$, this implies that
\[
\left|F(x, y)-\left[ax^{\alpha}+ay^{\alpha}+\psi_{n}(x+y)\right]\right|\leq
a_{n}(\alpha)\varepsilon_{n}+b_{n}(\alpha)\varepsilon_{2} .
\quad
\left(x, y\in ]0, n]\right)
\]
Finally, \eqref{Eq1.15}, the last inequality and the triangle
inequality imply that
\begin{multline*}
\left|f(x, y, z)-\left[ax^{\alpha}+ay^{\alpha}+az^{\alpha}+\psi_{n}(x+y+z)\right]\right| \\
\leq
\left|f(x, y, z)-F(x, y+z)-(y+z)^{\alpha}h\left(\frac{y}{y+z}\right)\right| \\
+\left|F(x, y+z)-ax^{\alpha}-a(y+z)^{\alpha}-\psi_{1}(x+y+z)\right| \\
+\left|(y+z)^{\alpha}h\left(\frac{y}{y+z}\right)-ay^{\alpha}-az^{\alpha}+a(y+z)^{\alpha}\right| \\
\leq (\varepsilon_{1}+\varepsilon_{2})+
(a_{n}(\alpha)\varepsilon_{1}+b_{n}(\alpha)\varepsilon_{n})+
(2^{\alpha}n^{\alpha}K(\alpha)(\varepsilon_{1}+\varepsilon_{2})) \\
=\left(1+a_{n}(\alpha)+2^{\alpha}n^{\alpha}K(\alpha)\right)\varepsilon_{1}+
\left(1+b_{n}(\alpha)+2^{\alpha+2}n^{\alpha}K(\alpha)\right)\varepsilon_{2}
\end{multline*}
for all $x, y, z\in ]0, n]$.
All in all, for all $n\in\mathbb{N}$, there exists a function
$\psi_{n}:]0, 3n]\rightarrow\mathbb{R}$ such that
\[
\left|f(x, y, z)-\left[ax^{\alpha}+ay^{\alpha}+az^{\alpha}+\psi_{n}(x+y+z)\right]\right|
\leq c_{n}(\alpha)\varepsilon_{n}+d_{n}(\alpha)\varepsilon_{2}
\]
holds for all $x, y, z\in ]0, n]$, where
\[
c_{n}(\alpha)=2+7\cdot 2^{\alpha}n^{\alpha}K(\alpha)
\quad
\text{and}
\quad
d_{n}(\alpha)=4+7\cdot 2^{\alpha+2}n^{\alpha}K(\alpha),
\]
as it was stated in our theorem.

\end{proof}

\section{Corollaries and open problems}

With the choice $\varepsilon_{1}=\varepsilon_{2}=0$ we get the general
solutions of equation \eqref{m_eq}, in the investigated cases.

\begin{cor}\label{C4.1}
Let $\alpha\in\mathbb{R}$, $\alpha\neq 1$ and
suppose that the function $f:\mathbb{R}^{3}_{+}\rightarrow\mathbb{R}$
is symmetric and satisfies
functional equation \eqref{m_eq} for all $x, y, z\in\mathbb{R}_{++}$.

Then, in case $\alpha\neq 0$, there exist $a\in\mathbb{R}$ and a function
$\varphi_{1}:\mathbb{R}_{++}\rightarrow\mathbb{R}$ such that
\[
f(x, y, z)=ax^{\alpha}+ay^{\alpha}+az^{\alpha}+\varphi_{1}(x+y+z)
\]
holds for all $x, y, z\in\mathbb{R}_{++}$.

In case $\alpha=0$, there exists a function $\varphi_{2}:\mathbb{R}_{++}\rightarrow\mathbb{R}$
such that
\[
f(x, y, z)=\varphi_{2}(x+y+z)
\]
is fulfilled for all $x, y, z\in\mathbb{R}_{++}$.
\end{cor}

In view of Corollary \ref{C4.1}., our theorem says that the modified entropy equation
is stable in the sense of Hyers and Ulam on its one-dimensional domain with the
multiplicative function $\mu(x)=x^{\alpha}$ ($\alpha\leq 0, x\in\mathbb{R}_{++}$).

In case $1\neq \alpha>0$ we obtain however that functional equation
\eqref{m_eq} is stable on every cartesian product of bounded real
intervals of the form $]0, n]^{3}$ , where $n\in\mathbb{N}$.
Nevertheless, an easy computation shows that
\[
\lim_{n\rightarrow +\infty}c_{n}(\alpha)=+\infty
\quad
\lim_{n\rightarrow +\infty}d_{n}(\alpha)=+\infty.
\quad
\left(1\neq \alpha>0\right)
\]

To the best of our knowledge, this is a new phenomenon in the
stability theory of functional equations.
Since we cannot prove the 'standard' Hyers--Ulam stability in this case,
the following problem can be raised.

\begin{Opp}
Let $\alpha, \varepsilon_{1}, \varepsilon_{2}\in\mathbb{R}$,
$\alpha>0, \varepsilon_{1}, \varepsilon_{2}\geq 0$, and
$f:\mathbb{R}^{3}_{+}\rightarrow\mathbb{R}$ be a function.
Assume that
\[
\left|f(x, y, z)-f(x, y+z, 0)-(y+z)^{\alpha}f\left(0, \frac{y}{y+z}, \frac{z}{y+z}\right)\right|\leq \varepsilon_{1}
\]
and
\[
\left|f(x, y, z)-f\left(\sigma(x), \sigma(y), \sigma(z)\right)\right|\leq \varepsilon_{2}
\]
holds for all $x, y, z\in\mathbb{R}_{++}$ and for all
$\sigma:\left\{x, y, z\right\}\rightarrow\left\{x, y, z\right\}$ permutations.

Is is true that there exists a solution of equation
\eqref{m_eq} $h:\mathbb{R}^{3}_{++}\rightarrow\mathbb{R}$ such that
\[
\left|f(x, y, z)-h(x, y, z)\right|\leq K_{1}\varepsilon_{1}+K_{2}\varepsilon_{2}
\]
holds for all $x, y, z\in\mathbb{R}_{++}$ with some $K_{1}, K_{2}\in\mathbb{R}$?
\end{Opp}

The second open problem that can be raised is the stability problem of
the modified entropy equation itself, i.e.,
equation \eqref{Eq_me_eq}.

\begin{Opp}
Let $\varepsilon_{1}, \varepsilon_{2}\geq 0$, $\mu:\mathbb{R}^{k}_{++}\rightarrow\mathbb{R}$
be a given multiplicative function, $f:\mathbb{R}^{3k}_{+}\rightarrow\mathbb{R}$.
Assume that
\[
\left|f(x, y, z)-f(x, y+z, 0)-\mu(y+z)f\left(0, \frac{y}{y+z}, \frac{z}{y+z}\right)\right|\leq \varepsilon_{1}
\]
and
\[
\left|f(x, y, z)-f\left(\sigma(x), \sigma(y), \sigma(z)\right)\right|\leq \varepsilon_{2}
\]
holds for all $x, y, z\in\mathbb{R}^{k}_{++}$ and for all
$\sigma:\left\{x, y, z\right\}\rightarrow\left\{x, y, z\right\}$ permutation.

Is is true that there exists a solution of equation
\eqref{Eq_me_eq} $h:\mathbb{R}^{3k}_{++}\rightarrow\mathbb{R}$ such that
\[
\left|f(x, y, z)-h(x, y, z)\right|\leq K_{1}\varepsilon_{1}+K_{2}\varepsilon_{2}
\]
holds for all $x, y, z\in\mathbb{R}^{k}_{++}$ with certain $K_{1}, K_{2}\in\mathbb{R}$?
\end{Opp}

\section*{Acknowledgment}
The author is beholden to Professor Gyula Maksa for his valuable advice
during the preparation of the manuscript.

\end{document}